\documentclass[11pt]{amsart}
\usepackage[margin=1.1in]{geometry} 

\usepackage{amsmath, amsfonts, amsthm, amssymb,amscd, enumerate,dsfont}
\usepackage{graphics}
\usepackage{color}
\usepackage[colorlinks, linkcolor=blue, citecolor=red, urlcolor=blue]{hyperref} 
\usepackage{wasysym}
  
\usepackage{comment}
\usepackage{amsthm,amssymb,amsmath,amsfonts}
\usepackage{mathtools} 
\usepackage{enumerate}
\usepackage[all]{xy}

\newtheorem{Th}{Theorem}[section]
\newtheorem{Prop}[Th]{Proposition}
\newtheorem{Lemma}[Th]{Lemma}
\newtheorem{Lemma-Definition}[Th]{Lemma-Definition} 
\newtheorem{Corollary}[Th]{Corollary}
\newtheorem{Cor}[Th]{Corollary}
\newtheorem{Conj}[Th]{Conjecture}
\newtheorem{Quest}[Th]{Question}

\newtheorem{Problem}[Th]{Problem}
\newtheorem*{Prop*}{Proposition}
\newtheorem*{Quest*}{Question}
\newtheorem*{Th*}{Theorem}
\newtheorem*{Conj*}{Conjecture}
\newtheorem*{Ack}{Acknowledgments}

\theoremstyle{definition} 
\newtheorem{Def}[Th]{Definition}

\newtheorem{Example}[Th]{Example}
\newtheorem*{Def*}{Definition}

\DeclarePairedDelimiter\abs{\lvert}{\rvert}

\renewcommand{\emptyset}{\varnothing}

\newcommand{\gem}{\geqslant}
\newcommand{\lem}{\leqslant}
\newcommand{\ON}{\operatorname}
\newcommand{\OL}{\overline}
\newcommand{\WT}{\widetilde}

\newcommand{\mult}{\operatorname{mult}}
\newcommand{\Aut}{\operatorname{Aut}}

\newcommand{\Supp}{\operatorname{Supp}}
\newcommand{\Pic}{\operatorname{Pic}}

\newcommand{\ord}{\operatorname{ord}}

\newcommand{\Cox}{\ON{Cox}}
\newcommand{\Cr}{\ON{Cr}}
\newcommand{\Gr}{\ON{Gr}}

\newcommand{\PSL}{\operatorname{PSL}}

\newcommand{\MC}{\mathds{C}}

\newcommand{\MP}{\mathds{P}}
\newcommand{\MQ}{\mathds{Q}}
\newcommand{\MZ}{\mathds{Z}}

\newcommand{\CM}{\mathcal{M}}

\newcommand{\CA}{\mathcal{A}}
\newcommand{\CS}{\mathcal{S}}
\newcommand{\CO}{\mathcal{O}}
\newcommand{\CE}{\mathcal{E}}

\makeatletter
\newcommand{\extp}{\@ifnextchar^\@extp{\@extp^{\,}}}
\def\@extp^#1{\mathop{\bigwedge\nolimits^{\!#1}}}
\makeatother

\author{I.~Krylov}
\title{Families of embeddings of the alternating group of rank $5$ into the Cremona group}
\date{}
\begin{document}

\address{ \emph{Igor Krylov}\newline \textnormal{Korea Institute for Advanced Study, 85 Hoegiro, Dongdaemun-gu, Seoul 02455, Republic of Korea \newline
\texttt{IKrylov@kias.re.kr}}}

\maketitle

\begin{abstract}
I study embeddings of the alternating group of rank five into the Cremona group of rank three.
I find all embeddings induced by $\CA_5\MQ$-del Pezzo fibrations and I study their conjugacy.
As an application, I show that there is a series of continuous families of pairwise non-conjugate embeddings of alternating group of rank five into $\Cr_3(\MC)$. 
\end{abstract}

\section{Introduction}
The \emph{Cremona group} $\ON{Cr}_n$ of rank $n$ is a group of birational transformations of $\MP^n$.
It is natural to study the group by studying its subgroups and finite subgroups in particular.
The classification of finite subgroups in $\Cr_2$ up to conjugacy is almost complete \cite{DI09}.
It is not feasible to achieve a classification for $\Cr_3$, nevertheless we can say something about its finite subgroups, for example we know that $\Cr_n$ is Jordan \cite{Jordan-Cremona}.

If we limit the group types, for example to $p$-subgroups (\cite{PrSh18}) or simple non-abelian, then the problem becomes more manageable.
The motivating problem for me is the following.

\begin{Problem} \label{MotivatingProblem}
Classify the embeddings of finite simple non-abelian subgroups of $\ON{Cr}_3$ up to conjugacy.
\end{Problem}

The isomorphism types of simple non-abelian groups were classified in \cite[Theorem~1.3]{Pr12}, the possibilities are: $\CA_5$, $\CA_6$, $\CA_7$, $\PSL_2(7)$, $\ON{SL}_2(8)$, and $\ON{PSp}_4(3)$.
In this paper I study the families of subgroups of $\ON{Cr}_3$.

\begin{Quest} \label{MainQuestion}
Let $G$ be a group.
Is there a continuous family of embeddings $G \hookrightarrow \Cr_n$ which are not pairwise conjugate to each other?
\end{Quest}

In general, one expects that the asnwer is positive for small groups and negative for big groups.
For example there are huge families of pairwise non-conjugate embeddings of $\MZ/2\MZ$ induced by Bertini and Geiser involutions \cite{Pr15}.
Also it was recently shown in \cite{ACPS19} that there is a continuous family of pairwise non-conjugate embeddings of $\CS_4$ into $\Cr_3$.

On the other hand, finite non-abelian groups are big, thus one expects that there should be no continuous families of pairwise non-conjugate embeddings of these groups.
Indeed, there are only six simple non-abelian subgroups of $\Cr_2$ up to conjugacy: three subgroups isomorphic to $\CA_5$, two subgroups isomorphic to $\ON{PSL}_2(7)$, and one subgroup isomorphic to $\CA_6$.
In dimension three we know by \cite[Theorem~1.5]{Pr12} that there are only finitely many embeddings for the three largest finite simple non-abelian subgroups.
I study the embeddings of the smallest finite simple non-abelian subgroup: $\CA_5$.

\begin{Th} \label{MainCor}
For each $k\gem 2$ there is a $2k-3$-dimensional family of embeddings $\CA_5 \hookrightarrow \Cr_3$ which are not pairwise conjugate to each other.
\end{Th}

The remaining cases are $\CA_6$ and $\PSL_2(7)$.
Some progress toward classification the embeddings of these groups has been made in \cite{CS14} and \cite{Kr18}, so far the results agree with the following expectation.

\begin{Conj}
The embeddings into $\Cr_3$ of $\CA_6$ and $\PSL_2(7)$ up to conjugation form discrete families.
\end{Conj}

\subsection{$G\MQ$-Mori fiber spaces}
The study of rational group actions on $\MP^3$ may be replaced with the study of regular group actions on suitable rational varieties.

\begin{Def}
I say that $\pi \colon Y \to Z$ (or simply $Y/Z$ if $\pi$ is understood) is a $G\MQ$-\emph{Mori fiber space} if 
\begin{itemize}
\item $\dim Z < \dim Y$ and $\pi$ has connected fibers,
\item
$G$-action on $Y$ is faithful and $Z$ admits a $G$-action such that the map $\pi$ is $G$-equivariant,
\item 
$Y$ is terminal and $G\MQ$-factorial, that is $G$-invariant divisors on $Y$ are $\MQ$-factorial,
\item
$-K_Y$ is $\pi$-ample and the relative $G$-invariant Picard rank is $\rho^G(Y/Z)=1$.
\end{itemize}
When $Y$ is Gorenstein I omit $\MQ$ and say that $Y/Z$ is a $G$\emph{-Mori fiber space}.
If $\dim Y - \dim Z = 2$, then I say that $Y/Z$ is a $G\MQ$\emph{-del Pezzo fibration}.
\end{Def}

By \cite[Proposition~1.2]{GFanoI} for a rational action of $G$ on a rationally connected variety $X$ there is a $G\MQ$-Mori fiber space $\pi \colon Y \to Z$ such that $X$ is $G$-equivariantly birational to $Y$, hence the $G$-action on $Y$ induces a conjugate embedding of $G$ into $\ON{Bir}(X)$.
The classification of subgroups of $\ON{Cr}_n$ up to conjugation is equivalent to the classification of the rational $G\MQ$-Mori fiber spaces up to $G$-equivariant birational equivalence.

Now I give examples of $\CA_5$-del Pezzo fibrations.

\begin{Example}
The projective plane with $\CA_5$-action induced by a non-trivial three-dimensional representation $W_3$ is an $\CA_5$-del Pezzo surface.
A del Pezzo surface $S_5$ of degree $5$ satisfies $\CA_5 \subset \Aut S_5$ and is $\CA_5$-minimal, hence it is an $\CA_5$-del Pezzo surface.

Let $S$ be a $G$-del Pezzo surface, then $S \times \MP^1/\MP^1$ is a $G$-del Pezzo fibration.
It follows that $\MP^2 \times \MP^1/\MP^1$ and $S_5 \times \MP^1/\MP^1$ are $\CA_5$-del Pezzo fibrations.
\end{Example}

\begin{Example} \label{Quadric-Examples}
Let $W_3$ and $W_3^\prime$ be the irreducible $3$-dimensional representations of $\CA_5$.
I consider $\CA_5$-linearized vector bundles defined as
\begin{equation*}
\CE_n = \big( W_3 \times \MP^1 \big) \oplus \CO_{\MP^1}(-n) \quad \text{and} \quad  \CE_n^\prime = \big(W_3^\prime \times \MP^1 \big) \oplus \CO_{\MP^1}(-n).
\end{equation*}
Let $T_n = \MP(\CE_n)$ and $T_n^\prime = \MP(\CE_n^\prime)$ and denote by $\pi_T$ and $\pi_{T^\prime}$ the corresponding projections onto $\MP^1$.

I introduce coordinates on $T_n$ (resp. $T_n^\prime$) as follows.
Let $u,v$ be the coordinates on the base $\MP^1$, $x,y,z$ be the coordinates on $W_3$ (resp. $W_3^\prime$), and let $w$ be the coordinate on the fiber of $\CO_{\MP^1}(-n) \to \MP^1$.
The degrees of the coordinates are given by
\[
\left(\begin{array}{ccccccc}
u & v & x & y & z & w   \\
0 & 0 & 1 & 1 & 1 & 1\\
1 & 1 & 0 & 0 & 0 & -n 
\end{array}\right)
\]
There is a unique $\CA_5$-invariant conic $\Delta$ on $\MP(W_3)$ (resp. $\MP(W_3^\prime)$). 
Up to a change of coordinates on $W_3$ (res. $W_3^\prime)$ I may assume that its equation is $xz - y^2 = 0$.
Let $X_n$ (resp. $X_n^\prime$) be the hypersurface in $T_n$ (resp. $T_n^\prime$) given by the equation
$$
a_{2n}(u,v) w^2 = xz - y^2,
$$
where $\deg a_{2n} = 2n$.
Observe that this equation is $\CA_5$-invariant, thus $X_n$ and $X_n^\prime$ admit a faithful $\CA_5$-action.
The $\CA_5$-varieties $X_n$ and $X_n^\prime$ are smooth if and only if $a_{2n}$ has no multiple factors.
If $l(u,v)$ is a multiple factor of $a_{2n}$, then $X_n$ and $X_n^\prime$ have a $cA_1$-singularity at $x=y=z=l(u,v)=0$.
We have $\rho(X_n) = \rho(X_n^\prime) = 2$ if $a_{2n}$ is not a square (Lemma \ref{Pic-Lemma}).
The restriction $\pi$ (resp. $\pi^\prime$) of $\pi_T$ (resp. $\pi_T^\prime$) to $X_n$ (resp. $X_n^\prime$) induces the structure of $\CA_5$-del Pezzo fibration on $X_n$ (resp. $X_n^\prime$).

Note that the varieties $X_1$ and $X_1^\prime$ are unique while for $n \gem 2$ the families of varieties $X_n$ and $X_n^\prime$ are of dimension $2n-3$.
\end{Example}

I show that these are essentially the only $\CA_5\MQ$-del Pezzo fibrations with the trivial action on the base.

\begin{Th} \label{Th-Classification}
Let $\pi: V \to \MP^1$ be an $\CA_5\MQ$-del Pezzo fibration and suppose $\CA_5$ acts trivially on the base.
Then one of the following holds:
\begin{enumerate}
\item 
The general fiber of $\pi$ is $\MP^2$ and $V$ is $\CA_5$-equivariantly birational to $\MP^2 \times \MP^1$ with $\CA_5$ acting only on the first factor;
\item
$V \cong_{\CA_5} S_5 \times \MP^1$, where $S_5$ is the del Pezzo surface of degree $5$ and $\CA_5$ acts only on the first factor;
\item
The general fiber of $\pi$ is a quadric and $V$ is $\CA_5$-equivariantly birational to a smooth $X_n$.
\end{enumerate}
\end{Th}


The varieties from the case \emph{(3)} are inducing families of embeddings of $\CA_5$ into $\Cr_3$.
In order to show that these embeddings are not pairwise conjugate I use the notion of $G$-equivariant birational superrigidity.


\begin{Def}
Let $\pi_X \colon X \to S$ and $\pi_Y \colon Y \to Z$ be $G\MQ$-Mori fiber spaces.
A $G$-equivariant birational map $\chi \colon X \dasharrow Y$ is called \emph{square} if it fits into a commutative diagram
\begin{displaymath}
\xymatrix
{ 
	X\ar@{-->}[r]^\chi \ar[d]_{\pi_X} & \ar[d]^{\pi_Y} Y \\
	S\ar@{-->}[r]^g & Z,
}
\end{displaymath}
where $g$ is birational and, in addition, the induced map on the generic fibers $\chi_\eta \colon X_\eta \dasharrow Y_\eta$ is isomorphism of $G$-varieties.
In this case I say that $X/S$ and $Y/Z$ are $G$-equivariantly square-birational.
\end{Def}

\begin{Def}
I say that $\pi_X \colon X \to S$ is $G$-\emph{equivariantly superrigid} if for any $G\MQ$-Mori fiber space $\pi_Y \colon Y \to Z$ any $G$-equivariant birational map $\chi \colon X \to Y$ is $G$-equivariantly square birational.
\end{Def}

\begin{Th} \label{Th-Conjugacy-classes}
Suppose $X_n$ (resp. $X_n^\prime$) is smooth and $n \gem 2$, then $X_n$ (resp. $X_n^\prime$) is $\CA_5$-equivariantly birationally superrigid.
\end{Th}


\begin{Cor}
Suppose $X_n$, $X_k$, $X_m^\prime$, and $X_l^\prime$ are smooth and $n,m,k,l\gem 2$, then:
\begin{enumerate}
\item
The varieties $X_n$ and $X_m^\prime$ are not $\CA_5$-equivariantly birational to $\MP^2 \times \MP^1$ and $S_5 \times \MP^1$;
\item
The variety $X_n$ is not $\CA_5$-equivariantly birational to $X_m^\prime$;
\item
The variety $X_n$ is $\CA_5$-equivariantly birational to $X_k$ if and only if $X_n$ is $\CA_5$-equivariantly isomorphic to $X_k$ and the same holds for $X_m^\prime$ and $X_l^\prime$;
\item
For $n \gem 2$ the family of $\CA_5$-del Pezzo fibrations $X_n$ (or $X_m^\prime)$ induces a $2n-3$-dimensional family (resp. $2m-3$-dimensional) of pairwise non-conjugate embeddings $\CA_5 \hookrightarrow \Cr_3$.
\end{enumerate}
\end{Cor}

The assertion \emph{(4)} is Theorem \ref{MainCor}.

\begin{Ack}
The author would like to thank Andrey Trepalin and Constantin Shramov for pointing out how to work with del Pezzo surfaces of degree $5$ over non-closed fields.
The author was supported by KIAS individual grant MG069801 at Korea Institute for Advanced Study.
\end{Ack}

\subsection*{Notations and conventions}
I work over $\MC$ unless stated otherwise.
All varieties are considered to be projective and normal.
I denote the symmetric and alternating groups of rank $n$ by $\CS_n$ and $\CA_n$ respectively.
I denote the del Pezzo surface of degree $5$ by $S_5$ and Clebsch cubic by $S_3$.
I denote the linear equivalence of divisors by $\sim$, the numerical equivalence of cycles by $\equiv$, and $\CA_5$-equivariant biregular equivalence by $\cong_{\CA_5}$.
Given a curve $\Delta$ I denote its orbit by $\OL \Delta$.
By a slight abuse of notations I denote the curve $\sum_{C \in \OL \Delta} C$ by $\OL \Delta$ as well.

\section{The represenatives of the $\CA_5$-equivariant birational classes of $\CA_5\MQ$-del Pezzo fibrations with the trivial action on the base} \label{Classification-section}

\subsection{Representations of $\CA_5$ and its central extension}

Recall that $\CA_5$ has the following irreducible representations: $I = I_1$, $W_3$, $W_3^\prime$, $W_4$, and $W_5$ (see \cite[Sections~5.2~and~5.6]{A5-Book}, for details).
The lower index is the dimension of the representation.
The action of $\CA_5$ on $\MP^1$ is induced by irreducible representations of the central extension $2.\CA_5$: $U_2$ and $U_2^\prime$.

I recall some facts about the action of $\CA_5$ on smooth curves and surfaces 

\begin{Lemma}[{ \cite[Lemmas~5.1.3,~5.1.4~and~Section~5.2]{A5-Book}}] \label{orbits}
Let $S$ be a smooth surface and $C$ be a smooth curve admitting a non-trivial action of $\CA_5$.
\begin{enumerate}
\item
Let $\Sigma$ be an $\CA_5$-orbit on $C$, then $\abs \Sigma = 12$, $20$, or $60$.
\item
Let $\Sigma$ be an $\CA_5$-orbit on $S$, then $\abs \Sigma \gem 5$.
\end{enumerate}
\end{Lemma}

\begin{Prop}[{\cite[Chapter~6]{A5-Book}}] \label{A5-fibers}
Let $\pi: X \to \MP^1$ be an $\CA_5\MQ$-del Pezzo fibration and let $F$ be a general fiber of $\pi$.
Then one of the following holds
\begin{enumerate}
\item [(1)]
$F \cong_{\CA_5} \MP(W_3)$ or $F \cong_{\CA_5} \MP(W_3^\prime)$,
\item [(2)]
$F \cong_{\CA_5} S_5$,
\item [(3)]
$F \cong_{\CA_5} S_3$, 
\item [(4a)]
$F \cong_{\CA_5} \MP(U_2) \times \MP(U_2)$ or
$F \cong_{\CA_5} \MP(U_2^\prime) \times \MP(U_2^\prime)$, in this case I say $F$ is a quadric with a diagonal $\CA_5$-action,
\item [(4b)]
$F \cong_{\CA_5} \MP(U_2) \times \MP(U_2^\prime)$, this case I say that $F$ is a quadric with a twisted diagonal $\CA_5$-action,
\item [(4c)]
$F \cong_{\CA_5} \MP(U_2) \times \MP(I \oplus I)$ or $F \cong_{\CA_5} \MP(U_2^\prime) \times \MP(I \oplus I)$, in this case I say that $F$ is a quadric with a one-factor $\CA_5$-action.
\end{enumerate}
\end{Prop}

Knowing the induced $\CA_5$-action on $\MP^3$ is very useful for studying cases \emph{(4a)}, \emph{(4b)}, and \emph{(4c)}.

\begin{Lemma} \label{quadric-embedding}
Let $F$ be a quadric with an action of $\CA_5$, let $V$ be the $4$-dimensional representation dual to $H^0(F, -\frac{1}{2} K_F)$, and let $g: F \hookrightarrow \MP(V)$ be the $\CA_5$-equivariant embedding induced by $\abs{ -\frac{1}{2} K_F}$.
Then
\begin{enumerate}
\item
the $\CA_5$-action on $F$ is diagonal if and only if $V \cong_{\CA_5} W_3 \oplus I$ or $V \cong_{\CA_5} W_3^\prime \oplus I$;
\item
the $\CA_5$-action on $F$ is twisted diagonal if and only if $V \cong_{\CA_5} W_4$;
\item
the $\CA_5$-action on $F$ is a one-factor action if and only if $\MP(V) \cong_{\CA_5} \MP(U_2 \oplus U_2)$, or $\MP(V) \cong_{\CA_5} \MP(U_2^\prime \oplus U_2^\prime)$.
\end{enumerate}
\end{Lemma}
\begin{proof}
Assertion (1) is equivalent to \cite[Lemma~6.3.3,~(i)]{A5-Book}.

Clearly, $\MP(V)$ has no invariant lines and planes if and only if $F$ has no invariant divisors of degree $(1,1)$ or $(0,1)$.
It follows that the action on $F$ is twisted diagonal if and only if $V \cong_{\CA_5} W_4$.

Suppose $F \cong_{\CA_5} \MP(U_2) \times \MP(I \oplus I)$.
Let $L_1$ and $L_2$ be the curves of bi-degree $(0,1)$ on $F$.
Then $L_1 \cong_{\CA_5} L_2 \cong_{\CA_5} \MP(U_2)$.
Since $g(L_1)$ and $g(L_2)$ are $\CA_5$-invariant skew lines we see that $V \cong U_2 \oplus U_2$.
The converse follows from the assertions \emph{(1)} and \emph{(2)}.
\end{proof}

\subsection{The general fiber is $\MP^2$}

\begin{Lemma} \label{deg9}
Let $\pi \colon X \to \MP^1$ be an $\CA_5\MQ$-del Pezzo fibration of degree $9$ and suppose $\CA_5$ acts trivially on the base.
Then $X$ is $\CA_5$-equivariantly birational to $\MP^2\times \MP^1$ with the trivial action of $\CA_5$ on $\MP^1$.
\end{Lemma}
\begin{proof}
Let $X_\eta$ be the generic fiber of $\pi$.
It is a form of $\MP^2$ and has a point, thus it is $\MP^2$.
The isomorphism $X_\eta \cong \MP^2$ induces the $\CA_5$-equivariant birational map to $\MP^2 \times \MP^1$.
\end{proof}

Note that $\MP^2 \times \MP^1$ is not the only $\CA_5\MQ$-del Pezzo fibration of degree $9$ with a trivial action on the base.

\begin{Example}
Let $\Delta$ be the $\CA_5$-invariant curve of degree $2$ in the fiber $F\subset \MP^2 \times \MP^1$ over a point $P \in \MP^1$.
Then I may blow up $\MP^2 \times \MP^1$ at $\Delta$ and then contract the proper transform of $F$ to acquire a variety $V_1$ with a $\frac{1}{2}(1,1,1)$-singularity.
The composition $\MP^2 \times \MP^1 \dasharrow V_1$ of these maps is an elementary $\CA_5$-equivariant Sarkisov link of $\CA_5\MQ$-del Pezzo fibrations.
The new fiber over $P$ is isomorphic to $\MP(1,1,4)$.
\end{Example}

\begin{Example}
Consider a toric variety $R_n$ such that $\Cox R_n = \MC[u,v,x,y,z,w]$, where the irrelevant ideal $I = \langle u,v\rangle \cap \langle x,y,z,w \rangle$ and the grading given by
\[
\left(\begin{array}{ccccccc}
u & v & x & y & z & w   \\
0 & 0 & 1 & 1 & 1 & 2\\
1 & 1 & 0 & 0 & 0 & -n 
\end{array}\right)
\]
I assume that $\CA_5$-action on $R_n$ comes from the identification of $\MC^3_{x,y,z} \cong W_3$.
Suppose that the equation of the $\CA_5$-invariant quadric on $W_3$ is $xz - y^2 = 0$.
Then consider an $\CA_5$-invariant hypersurface $V_n$ in $R_n$ given by the equation
$$
a_n (u,v) w = xz-y^2.
$$
There is a natural $\CA_5$-equivariant projection $\pi_R \colon R_n \to \MP^1_{u,v}$ which is a $\MP(1,1,1,2)$-bundle.
The restriction $\pi = \pi_R \vert_{V_n} \colon V_n \to \MP^1_{u,v}$ is an $\CA_5\MQ$-del Pezzo fibration of degree $9$ with the trivial action on the base.
The variety $V_n$ has $2$-Goreinstein singularities at $x=y=z=a_n(u,v)=0$.
For every simple root of $a_n(u,v)$ there is a $\frac{1}{2}(1,1,1)$-singularity and for a root of multiplicity $k$ there is a $cA_1 /\mu_2$-singularity, where $\mu_2$ is the cyclic group of order $2$.

\end{Example}

\begin{Quest}
Let $\pi \colon X \to \MP^1$ be an $\CA_5\MQ$-del Pezzo fibration and suppose that the general fiber of $\pi$ is $\MP^2$.
Suppose $\CA_5$ acts trivially on the base.
Is it true that $X \cong V_n$ for some $a_n(u,v)$?
\end{Quest}


\subsection{The general fiber is $S_5$}

\begin{Prop} \label{deg5}
Let $\pi: X \to \MP^1$ be an $A_5\MQ$-del Pezzo fibration of degree $5$.
Then it is $\CA_5$-equivariantly birational to $S_5 \times \MP^1$.
\end{Prop}
\begin{proof}
Consider the generic fiber $X_\eta$ of $\pi$.
It is a del Pezzo surface of degree $5$ over $\MC(t)$ which admits an action of $\CA_5$.
Recall, that a del Pezzo surface of degree $5$ over an algebraically closed field is unique and it admits a faithful $\CA_5$-action.
Hence the $\CA_5$-action on $\Pic(\OL{X_\eta})$ is faithful and induces an embedding $\CA_5 \hookrightarrow \ON{W}(A_4)\cong \CS_5$, where $A_4$ is the root lattice.

Let $\Gamma = \ON{Gal}(\OL{\MC(t)} / \MC(t))$, then the action of $\Gamma$ and $\CA_5$ on $\OL X_\eta$ commutes.
It follows that $\Gamma$ commutes with $\ON{W}(A_4) = \mathcal{S}_5$, therefore all $(-1)$-curves of $\OL X_\eta$ are defined over $\MC(t)$.
Thus $X_\eta$ is acquired by a blow up of $\MP^2$ at $4$ points defined over $\MC(t)$.
In particular a del Pezzo surface of degree $5$ with an $\CA_5$-action is unique.

On the other hand the generic fiber of the projection $S_5 \times \MP^1 \to \MP^1$ is a del Pezzo surface of degree $5$ with an action of $\CA_5$, hence isomorphic to $X_\eta$.
The isomorphism induces the $\CA_5$-equivariant birational map $X \dasharrow S_5 \times \MP^1$.

\end{proof}

Unlike the degree $9$ case, one can show that there are no other $\CA_5\MQ$-del Pezzo fibrations of degree $5$.

\begin{Lemma}
The threefold $S_5 \times \MP^1$ is the unique $\CA_5\MQ$-del Pezzo fibration of degree $5$.
\end{Lemma}
\begin{proof}
It is well known that $\ON{lct}_{\CA_5} S_5 \gem 1$, thus the statement follows from \cite[Theorem~1.5]{Ch-Cubics} (see \cite[Theorem~3.6]{Kr18} for $G$-invariant version).
\end{proof}

I expect that $S_5 \times \MP^1$ is the unique $\CA_5\MQ$-del Pezzo fibration in its birational class but it is $\CA_5$-equivariantly birational to other $\CA_5$-Mori fiber spaces.
We can easily construct many $\MP^1$-bundles $\CA_5$-equivariantly birational to $S_5 \times \MP^1$ by blowing up orbits of fibers of the projection onto $S_5$.

\begin{Quest}
Does there exist an $\CA_5\MQ$-Mori fiber space 
which is not $\CA_5$-equivariantly square birational to $S_5 \times \MP^1 /\MP^1$ or $S_5 \times \MP^1 / S_5$ but is $\CA_5$-equivariantly birational to $S_5 \times \MP^1$?
\end{Quest}

\subsection{The general fiber is $S_3$}

\begin{Lemma} \label{deg3}
Let $X_\eta$ be a cubic over $\MC(t)$ admitting a faithful $\CA_5$-action, then $\rho^{\CA_5} (X_\eta) = 2$.
In particular, the $\CA_5\MQ$-del Pezzo fibrations of degree $3$ do not exist.
\end{Lemma}
\begin{proof}
There is the unique $\CA_5$-invariant cubic $S_3$ in $\MP^3$: Clebsh cubic.
It follows that there is a unique $\CA_5$-invariant cubic $X_\eta$ in $\MP^3_{\MC(t)}$.
I may realize $X_\eta$ as the generic fiber of the projection $S_3 \times \MP^1 \to \MP^1$.
Thus $\rho^{\CA_5} (X_\eta) = 2$.
\end{proof}

\subsection{The general fiber is a quadric}
Recall that there are three types of $\CA_5$-action on $\MP^1 \times \MP^1$ as defined in Proposition \ref{A5-fibers}: diagonal, twisted diagonal, and one-factor action.
The $\CA_5$-quadric fibration exist only for the former type.

\begin{Lemma} \label{non-diag}
Let $\pi: X \to \MP^1$ be an $\CA_5$-equivariant map such that the $\CA_5$-action on a general fiber $F$ of $\pi$ is either twisted diagonal or a one-factor action.
Let $X_\eta$ be the generic fiber of $\pi$, then $\rho^{\CA_5}(X_\eta) = 2$.
\end{Lemma}
\begin{proof}
Suppose the $\CA_5$-action on $F$ is twisted diagonal, then $F \subset \MP(W_4)$ by Lemma \ref{quadric-embedding}.
Since $\ON{Sym}^2 (W_4) \cong W_5 \oplus W_4 \oplus I$, there is the unique $\CA_5$-invariant quadric $S_\eta$ in $\MP(W_4 \otimes \MC(t))$.
On the other hand, $S_\eta$ is also the generic fiber of the projection 
$$
\MP(U_2) \times \MP(U_2^\prime) \times \MP(I \oplus I) \to \MP(I \oplus I),
$$
thus we see that $\rho^{\CA_5} (S_\eta) = 2$.

Suppose the $\CA_5$-action on $F$ is a one-factor action corresponding to the representation $U_2$ (or $U_2^\prime$) of $2.\CA_5$.
Similarly to the twisted diagonal case I only need to show that there is a unique $\CA_5$-invariant quadric in $\MP(U_2 \oplus U_2)$ (resp. $\MP(U_2^\prime \oplus U_2^\prime)$).
There are $\CA_5$-invariant skew lines $Z_1$ and $Z_2$ and any $\CA_5$-invariant quadric must contain both $Z_1$ and $Z_2$ since there are no orbits of length $\lem 2$ on $\MP^1$.
On the other hand, the space of quadrics containing a pair of $\CA_5$-invariant skew lines is isomorphic to $\MP(W)$ for some $4$-dimensional representation $W$ of $\CA_5$ or $2.\CA_5$.
The image on $\MP(V)$ of the quadric with a one-factor action corresponds to an $\CA_5$-fixed point on $\MP(W)$.
On the other hand, the family of reducible quadrics containing $Z_1$ and $Z_2$ is of dimension $3$ and is $\CA_5$-invariant.
It follows that $W \cong_{\CA_5} W_3 \oplus I$ or $W \cong_{\CA_5} W_3^\prime \oplus I$. 
Thus the $\CA_5$-invariant quadric on $\MP(U_2 \oplus U_2)$ (resp. $\MP(U_2^\prime \oplus U_2^\prime)$) is unique.
\end{proof}

Recall the construction of the $\CA_5\MQ$-quadric fibrations $X_n$ and $X_n^\prime$ from Example \ref{Quadric-Examples}.

\begin{Prop} \label{diag}
Let $\pi_V \colon V \to \MP^1$ be an $\CA_5\MQ$-quadric fibration.
Suppose the action on a general fiber $F$ of $\pi_V$ is diagonal.
Then $V$ is $\CA_5$-equivariantly birational to a smooth $X_n$ or $X_n^\prime$ for some $a_{2n}(u,v)$.
\end{Prop}
\begin{proof}
The fiber $F$ embeds into $\MP(W_3 \oplus I)$ or $\MP(W_3^\prime \oplus I)$ by Lemma \ref{quadric-embedding}. 
The two cases are analogous, suppose the former, then
$$
X_\eta \hookrightarrow \MP\big( (W_3 \oplus I) \otimes \MC(t) \big).
$$
Let $x,y,z$ be the coordinates on $W_3$ and let $w$ be the coordinate on $I$.
Then, up to a change of coordinates on $W_3$, every $\CA_5$-invariant quadric in $\MP(W_3 \oplus I)$ has the equation
$$
\lambda(t) w^2 = \mu(t) (xz - y^2).
$$
It follows that up to a change of coordinates in $\MP\big( (W_3 \oplus I) \otimes \MC(t) \big)$ the equation of $X_\eta$ is
$$
b(t) w^2 = xz - y^2,
$$
where $b(t)$ has no multiple roots.

Set $a(u,v)$ to be the homogeneous polynomial of degree $2n$ without multiple factors satisfying $a(t,1) = b(t)$.
Let $X_n$ be the hypersurface in $T_n$ given by
$$
a(u,v) w^2 = xz - y^2.
$$
Then the general fiber of $X_n \to \MP^1$ is $\CA_5$-equivariantly isomorphic to $X_\eta$.
This isomorphism of the general fibers induces the $\CA_5$-equivariant birational map $V \dasharrow X_n$.
\end{proof}

This completes the proof of Theorem \ref{Th-Classification}.

\section{Rigidity results} \label{rigidity-section}

This section is devoted to proving Theorem \ref{Th-Conjugacy-classes}.
First, I present some elementary results on the geometry of varieties $X_n$.
Next I recall Noether-Fano method and the notion of maximal centers.
Then I show that the only possible maximal centers are the points on the $\CA_5$-invariant curve of degree $2$.
At last, I use the technique of supermaximal singularities \cite{Pukh123} to prove $\CA_5$-equivariant birational superrigidity of $X_n$ and $X_n^\prime$ for $n \gem 2$.

\subsection{The geometry of $X_n$}

Note that $T_n \cong T_n^\prime$ and for a fixed $a_{2n}$ we have $X_n \cong X_n^\prime$.
These isomorphisms are not $\CA_5$-equivariant.
From now on I work with $T_n$ and $X_n$, the proofs for $X_n^\prime$ are identical.
In this section I do not assume that $a_{2n}(u,v)$ has no multiple factors.
In that case $X_n$ is no longer smooth.
Indeed, for each multiple linear factor $l(u,v)$ of $a_{2n}(u,v)$ the point $(x=y=z=l(u,v)=0)$ is a $cA_1$-singularity.

Denote by $H_T$ the divisor class of $(x=0)$ and by $F_T$ the divisor class of $(u=0)$ on $T_n$.
Denote $H = H_T \vert_{X_n}$ and $F = F_T \vert_{X_n}$.
Let $s_T \in A^3(T_n)$ and $s \in A^2(X_n)$ be the classes of the curve $x=y=w=0$.
Let $f_T \in A^3(T_n)$ and $f \in A^2(X_n)$ be the classes of the curve $x=y=u=0$.
Using the fact that 
$$
(x=y=z=w=0) = \emptyset \quad \text{and} \quad (x=y=z=u=0) = \ON{pt}
$$ 
I compute the intersections on $T_n$.

\begin{Lemma}\label{Toric-info}
The following holds for $T_n$:
\begin{enumerate}
\item
The classes $s_T$ and $f_T$ generate the cone of effective curves on $T_n$;
\item
$H_T^2 \cdot F_T \equiv f_T$;
\item
$H_T^3 \equiv s_T + nf_T$;
\item
$H_T^3 \cdot F_T = 1$;
\item
$H_T^4 = n$.
\end{enumerate}
\end{Lemma}

\begin{Corollary} \label{Xn-Int}
The following holds for $X_n$:
\begin{enumerate}
\item
$H \cdot F \equiv 2f$;
\item
$H^2 \equiv 2s + 2nf$;
\item
$H^2 \cdot F = 2$;
\item
$H^3 = 2n$;
\item
$ K_{X_n} \sim -2H + (n-2) F$;
\item
$K_{X_n}^2 \equiv 8s + 24f - 8 n f$.
\end{enumerate}
\end{Corollary}

In order to better understand a singular $X_n$, it is useful to know how is it related to a smooth $X_m$.

\begin{Lemma} \label{Link}
Let variety $X_n$ correspond to $a_{2n}(u,v)$ and $X_{n+1}$ to $u^2a_{2n}(u,v)$.
Denote the fiber $u=0$ by $F$.
Let $\sigma \colon \WT X \to X_n$ be the blow up of $X_n$ along the curve $u=w=0$.
Then there is a map $\psi \colon \WT X \to X_{n+1}$ contracting $\sigma^{-1} F$ to a $cA_1$ singularity $x=y=z=u=0$.
The composition $\varphi_u = \psi \circ \sigma^{-1}$ is an $\CA_5$-equivariant elementary Sarkisov link.

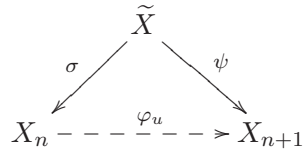
\begin{displaymath}
\xymatrix{
  & \WT X \ar[dl]_{\sigma} \ar[dr]^{\psi} & \\
X_n \ar@{-->}[rr]^{\varphi_u} & & X_{n+1} 
}
\end{displaymath}

\end{Lemma}
\begin{proof}
Elementary calculations.
\end{proof}

Thus we see that any singular $X_n$ is $\CA_5$-equivariantly square birational to a smooth $X_m$ with $m = n - k$ for some $k >0$.

\begin{Lemma} \label{Pic-Lemma}
Suppose $a_{2n}(u,v)$ is not a square, then
\begin{enumerate}
\item
$\Pic X_n \cong \MZ \cdot H \oplus \MZ \cdot F$,
\item
The classes $s$ and $f$ generate the cone of effective curves $\ON{NE} (X_n)$.
\end{enumerate}
\end{Lemma}
\begin{proof}
The assetion \emph{(1)} holds since $\Pic X_n \cong \Pic X_\eta \oplus \Pic \MP^1$ and $\Pic X_\eta$ is generated by the hyperplane section if and only if $a_{2n}(u,v)$ is not a complete square.

The assertion \emph{(2)} follows from \emph{(1)} and Lemma \ref{Toric-info} \emph{(1)} for smooth $X_n$ by Poincare duality.
Any $X_n$ is related to some smooth $X_m$ by a sequence of elementary Sarkisov links described in Lemma \ref{Link}.
The elementary links preserve the dimension of $\ON{NE}(X_n)$, hence the assertion holds by Lemma \ref{Toric-info} \emph{(1)}.
\end{proof}

\begin{Lemma} \label{K-cond}
Suppose $n \gem 2$, then $X_n$ satisfies the $K$-condition, that is $-K_{X_n}$ is not in the interior of the cone of movable divisors.
\end{Lemma}
\begin{proof}
The linear system $\abs H$ defines a divisorial contraction $\sigma \colon X_n \to Y_n$, where
$Y_n$ is a hypersurface in $\MP(1,1,n,n,n)$ given by the equation
$$
a_{2n}(u,v) = xz - y^2.
$$
Thus the cone of movable divisors of $X_n$ is generated by $H$ and $F$ which implies the statement of the lemma by Corollary \ref{Xn-Int} \emph{(5)}.
\end{proof}

\subsection{Noether-Fano method}
For definitions of canonical singularities of pairs we refer the reader to \cite[pages~16-17]{SingPairs} and \cite[Definition~2.1]{Corti00}.
Let $\pi: X \to \MP^1$ be an $\CA_5$-del Pezzo fibration.

Suppose we are given a birational map $\chi \colon X \dasharrow Y$ to a Mori fiber space $\pi_Y \colon Y \to Z$.
Let $\CM_Y$ be a very ample complete linear system on $Y$.
I say $\CM= \chi^{-1} \CM_Y$ is a \emph{mobile linear system associated to} $\chi$.
There are numbers $\lambda \in \MQ_{+}$ and $l \in \MQ$ such that $\lambda \CM \sim -K_X + lF$.
The Noether-Fano inequality is the essential result used to prove birational rigidity-type results.

\begin{Th}[Noether-Fano inequality, {\cite[Theorem~4.2]{Corti95}}]
Suppose $l>0$ and $(X,\lambda \CM)$ has canonical singularities, then $\chi$ is isomorphism.
\end{Th}

Let $E$ be a divisorial valuation of $\MC(X)$.
If $(E,X,\lambda \CM) < 0$, then I say that $E$ is a \emph{maximal singularity} of the pair $(X,\lambda \CM)$.
I call the center $Z$ of $E$ on $X$ a \emph{maximal center} of the pair $(X,\lambda \CM)$.

Now I examine, which subvarieties $Z \subset X$ can be maximal centers.

Set $X = X_n$ for some $a_{2n}(u,v)$ that is not a square and let $\chi$, $\CM$, $\lambda$, $l$ be as above.

Let $C$ be an irreducible curve on $X$, I say $C$ is \emph{horizontal} if $\pi(C) = \MP^1$ and \emph{vertical} if $\pi(C)$ is a point.
I say that a curve $C$ is \emph{horizontal} (resp. \emph{vertical}) if every irreducible component of $C$ is horizontal (resp. vertical).

Let $C \subset X$ be a horizontal or a vertical curve. 
I define its degree as follows
$$
\deg(C) = 
\begin{cases}
	-K_X \cdot C / 2, & \text{if } C \text{ is vertical}\\
	C \cdot F, & \text{if } C \text{ is horizontal}
\end{cases}
$$

\begin{Lemma}\label{Curves}
Suppose $(X, \lambda \CM)$ is not canonical at a curve $\Delta$, then $\Delta$ is an $\CA_5$-invariant vertical curve of degree $2$.
\end{Lemma}
\begin{proof}
Let $\OL \Delta$ be the $\CA_5$-orbit of $\Delta$.
The pair $(X, \lambda \CM)$ is not canonical at a curve $\Delta$ if and only if $\mult_\Delta \lambda \CM > 1$ and hence if and only if $\mult_{\OL \Delta} \lambda \CM > 1$.

Suppose $\Delta$ is horizontal, let $F$ be a general fiber of $\pi$ and let $D_1, D_2 \in \CM$ be general divisors.
Then the set-theoretic intersection 
$$
\Sigma = F \cap \OL\Delta \subset F \cap D_1 \cap D_2
$$ 
is a union of orbits on $F$.
Hence
$$
8 = F \cdot \lambda D_1 \cdot \lambda D_2 > \lambda^2 F \cdot \OL \Delta \gem \big| \Sigma \big| \gem 12,
$$
a contradiction.

Suppose $\OL \Delta \subset F$ and let $D$ be general in $\CM$.
Let $H_F = H \vert_F$, then $D \vert_F \sim 2H_F$ and $\ON{ord}_{\OL \Delta} D \vert_F >1$.
It follows that $\deg \OL \Delta \lem 4$ and \cite[Lemma~6.4.4]{A5-Book} implies that $\OL \Delta = \Delta$ is the $\CA_5$-invariant curve of degree $2$.

Suppose $F$ is singular, then it is a cone over $\MP^1$ with the $\CA_5$-action inherited from $\MP^1$.
Let $C$ be the unique $\CA_5$-invariant curve of degree $2$ on $F$.
Suppose $\Delta \neq C$ and denote $\Delta \sim k H_F$. 
Set $\Sigma = \Delta \cap C$, it is a union of $\CA_5$-orbits on $C$ hence $\abs{\Sigma} \gem 12$ by Lemma \ref{orbits}.
It follows that $\Delta \cdot C \gem 12$ and $k \gem 6$.
On the other hand, $D \vert_F \sim 2H_F$, a contradiction.
\end{proof}

\begin{Lemma} \label{FixPt}
Let $P$ be an $\CA_5$-invariant point and suppose $X$ is smooth at $P$.
Then pair $(X, \lambda \CM)$ is canonical at $P$.
\end{Lemma}
\begin{proof}
Suppose $(X,\lambda \CM)$ is not canonical at $P$ and let $E_\infty$ be the divisorial valuation over $X$ such that $a(E_\infty,X,\lambda \CM) <0$.

First, observe that a fiber $F$ of $\pi$ containing $P$ is a quadratic cone and $P$ is its vertex.
Let $\sigma \colon \WT X \to X$ be the blow up at $P$ and let $E$ be the exceptional divisor of $\sigma$.
Let $L$ be a general line through $P$, then for general $D \in \CM$
$$
\mult_P \lambda D \lem L \cdot \lambda D = 2.
$$
It follows that $a(E,X,\lambda\CM) \gem 0$, hence the center $B$ of $E_\infty$ on $\WT X$ is a point or a curve on $E$.

Note that the action of $\CA_5$ on $E$ is non-trivial.
Indeed, the point $P$ up to a change of coordinates on $\MP^1_{u,v}$ has the equations $u=x=y=z=0$ and the local equation of $X$ near $P$ is $u=0$, thus $E \cong_{\CA_5} \MP(W_3)$.
Denote $\WT \lambda \CM_E = (\sigma^{-1} \CM )\vert_E$, then $\ord_B \lambda \WT \CM_E > 1$ and $\deg \WT \CM_E = \mult_P \lambda \CM \lem 2$.
On the other hand, if $B$ is a curve, then $\deg \OL B \gem 2$, a contradiction.

Suppose $B$ is a point and let $\OL B$ be the $\CA_5$-orbit of $B$.
Then $\big| \OL B \big| \gem 6$ and there are $4$ points $P_1,P_2,P_3,P_4\in \OL B \subset E \cong \MP^2$ in general position.
I claim that
\begin{equation}
\label{local-equation-1}
\sum_{i=1}^4 \mult_{P_i} C \lem 2 \deg C 
\end{equation}
for any curve $C \subset E$.
Indeed, denote by $L_{ij}$ the line on $E$ passing through $P_i$ and $P_j$.
Decomposing $C = C^\prime + \sum \alpha_{ij} L_{ij}$ and counting multiplicities I conclude the inequality (\ref{local-equation-1}).
Thus 
$$
4 < 4 \mult_B \lambda \CM_E = \sum_{i=1}^4 \mult_{P_i} \lambda \CM_E \lem 4,
$$
a contradiction.
\end{proof}

It is possible that the pair $(X,\lambda \CM)$ is not canonical at $\CA_5$-invariant curves of degree $2$.
On the other hand, the elementary Sarkisov links originating from them are described in Lemma \ref{Link}, these links are $\CA_5$-equivariant square birational maps.
Using these links I acquire a new pair which is canonical at all curves.

Let $b(u,v)$ be a polynomial of degree $k$.
Then there is the associated map $\varphi_b \colon X \dasharrow X_b$, where $X_b$ is a hypersurface in $T_{n+k}$ given by the equation $a(u,v)(b(u,v))^2 w^2 = xz - y^2$.
The map $\varphi_b$ is the composition of elementary links described in Lemma \ref{Link}.
Denote $\CM_b = \varphi_b^{-1} \CM$.

\begin{Prop} \label{Untwisting}
Suppose $X$ is smooth, then there is $b(u,v)$ such that the pair $(X_b, \lambda \CM_b)$ is canonical at curves and $\CA_5$-invariant points.
\end{Prop}
\begin{proof}
I prove the proposition by playing the two-ray game.
Suppose $(X,\lambda \CM)$ is not canonical at a curve $\Delta_1$.
By Lemma \ref{Curves} the curve $\Delta_1$ is of degree $2$ and is $\CA_5$-invariant, hence its equations are $l_1(u,v) = w = 0$ for some linear $l_1(u,v)$.
The elementary Sarkisov link starting at $\Delta_1$ is the map $\varphi_{l_1}$.
Let $\CM_{l_1} = \varphi_{l_1}^{-1} \CM$.
If the pair $(X_{l_1}, \CM_{l_1})$ is canonical at curves, then I am done.
Otherwise there is a curve $\Delta_2$ and an elementary Sarkisov link $\varphi_{l_2} \colon X_{l_1} \dasharrow X_{l_1l_2}$, and I repeat the process as many times as required.
The process terminates by \cite[Theorem~6.1]{Corti95} and I set $b = l_1l_2\dots l_k$.
The pair $(X_b, \lambda \CM_b)$ is canonical at curves by construction and Lemma \ref{Curves}.

The pair $(X_b, \lambda \CM_b)$ is canonical at smooth $\CA_5$-invariant points by Lemma \ref{FixPt}.
I will now show that the pair $(X_b, \lambda \CM_b)$ is canonical at singular $\CA_5$-invariant points as well.

Recall that by \cite[Theorem~1.1]{KcA1} if the pair $(X_b, \lambda \CM_b)$ is not canonical at a $cA_1$-point $P$ with the local equation
$$
xz - y^2 - u^N = 0,
$$
then either $a(E_{s,t},X_b,\lambda \CM_b) <0$, where $E_{s,t}$ is the exceptional divisor of a $(s,t,2t-s,1)$-weighted blow up at $P$ for some coprime $s\lem t \lem N/2$ or $N=3$ and $a(E_{sp},X_b,\lambda \CM_b) <0$, where $E_{sp}$ is the exceptional divisor of a $(1,3,5,2)$-weighted blow up if $N = 3$.

Note that $a(E_{1,1},X_b,\lambda \CM_b) > 0$ by construction.
Thus, if $N = 2$, we are done.
We proceed by induction.
If $N=3$, then $a(E_{sp},X_b,\lambda \CM_b) <0$  and the pair $(X_{b/u},\lambda \CM_{b/u})$ is not canonical at the $\CA_5$-invariant point in the fiber $u=0$, which contradicts Lemma \ref{FixPt}.

Similarly if $N \gem 4$, then $a(E_{s,t},X_b,\lambda \CM_b) \gem 0$ for $s\gem 2$. Indeed, otherwise the pair $(X_{b/u},\lambda \CM_{b/u})$ is not canonical at the $\CA_5$-invariant point in the fiber $u=0$, which contradicts the assumption of induction.
Thus I may assume that $s = 1$ and $t > 1$.
Let $\sigma \colon \WT X_b \to X$ be the blow up at $P$ and let $\WT \CM_b = \sigma^{-1} \CM_b$.
Then the pair $(\WT X_n, \lambda \WT \CM_b)$ is not canonical at a line $L$ on the exceptional divisor $E$ of $\sigma$.
Hence, it is not canonical at each line in the orbit of $L$.
It follows that $\deg \WT \CM_b \vert_{E} > 6$ since the length of the orbit of $L$ is at least $12$, which contradicts $a(E,X_b,\CM_b) > 0$.
\end{proof}

It remains to show that the points that are not fixed by the $\CA_5$-action cannot be maximal centers.

\subsubsection{Orbits of points as non-canonical centers}
Let $P$ be a maximal center of the pair $(X, \lambda \CM)$.
I have already shown, that $P$ is a point which is not fixed by $\CA_5$-action.
In this subsection I show that $P$ must lie on an $\CA_5$-invariant curve of degree $2$.

Denote by $F$ the fiber containing $P$ and denote by $\OL P$ the $\CA_5$-orbit of $P$, then $(X, \lambda \CM)$ is not canonical at any point $P_i \in \OL P$.
Let $\Delta$ be the $\CA_5$-invariant curve of degree $2$ in the fiber containing $P$.

\begin{Lemma} \label{good-curves}
Suppose $P \not\in \Delta$, then one of the following holds:
\begin{enumerate}
\item
There is an irreducible curve $\Gamma$ of degree $4$ and distinct points $P_1,\dots,P_8 \in \OL P \cap \Gamma$ such that $\Gamma$ is smooth at $P_1,\dots,P_8$;
\item
There are smooth disintct irreducible curves $\Gamma_1,\Gamma_2$ of degree $2$, distinct points $P_1,\dots,P_4 \in \OL P \cap \Gamma_1$, and distinct points $P_5,\dots,P_8 \in \OL P \cap \Gamma_2$;
\item
The fiber $F$ is smooth, there is a smooth irreducible curve $\Gamma$ of bi-degree $(2,1)$ and there are distinct points $P_1,\dots, P_7 \in \OL P \cap \Gamma$.

\item
The fiber $F$ is singular, there are disinct smooth irreducible curves $\Gamma_1,\Gamma_2$ of degree $3$, there are distinct points $P_1,\dots,P_7 \in \Gamma_1 \cap \OL P$, and  there are distinct points $P_1^\prime,\dots,P_7^\prime \in \Gamma_2 \cap \OL P$.
\end{enumerate}
\end{Lemma}
\begin{proof}
For any $8$ points on $F$ there is a unique quadric section passing through them.
Let $C$ be the quadric section through $P_1, \dots, P_8$, note that $\deg C = 4$.

First, suppose $C$ is reducible, then $C$ has at most $3$ components.
If $C$ has $3$ components, then there is a conic $C$ containing at least $6$ points among $P_1,\dots,P_8$.
There is $g \in \CA_5$ such that $gC \neq gC$, thus I set $\Gamma_1 = C$ and $\Gamma_2 = gC$.

If $C$ is non-reduced, then simiarly to the previous case I set $\Gamma_1 = C/2$ and $\Gamma_2 = g\Gamma_1$ for some $g$ such that $g\Gamma_1 \neq \Gamma_1$.

Suppose $C = C_1 + C_2$ where $C_1,C_2$ are irreducible.
If $\deg C_1 = \deg C_2 = 2$, then I set $\Gamma_i = C_i$.
If $\deg C_1 = 1$ and $\deg C_2 = 3$, then $C_1$ contains at most one points among $P_1, \dots P_8$.
Thus I may assume that $P_1,\dots,P_7 \in C_2$.
If $F$ is smooth, then it is the situation \emph{(3)}.
For singular $F$ set $\Gamma_1 = C_2$ and $\Gamma_2 = g \Gamma_1$ for some $g$ such that $g\Gamma_1 \neq \Gamma_1$.

Pick a point $P_9 \in \OL P$ distinct from $P_1,\dots,P_8$.
Let $C_i$ be the quadric section through points $P_1,\dots,P_{i-1},P_{i+1},\dots,P_9$.
I may assume that all $C_i$ are irreducible, otherwise one of the previous cases applies.
I may also assume that every $C_i$ is singular at one of the points $P_j$, otherwise I am done.
Note that $C_i$ is singular at one point at most, hence if all $C_i$ coincide, I am also done.
Since $C_9 \cdot C_i = 8$ the curve $C_i$ must be singular at $P_9$ for $i\neq 9$.
But then $C_1 \cdot C_2 \gem 10$, a contradiction.
\end{proof}

\begin{Corollary}
Suppose $(X,\lambda\CM)$ is not caninical at a point $P$, then $P \in \Delta$.
\end{Corollary}
\begin{proof}
The proof is the case by case analysis for the curves $\Gamma$, $\Gamma_1$, and $\Gamma_2$ from Lemma \ref{good-curves}.
The case \emph{(1)} is analogous to the case \emph{(3)} and the case \emph{(4)} is analogous the case \emph{(2)}.

Suppose there is a curve $\Gamma$ as in the case \emph{(3)}. 
Let $D$ be a general divisor in $\CM$ and denote $D_F = D\vert_F$.
Then $\mult_{P_i} \lambda D_F > 1$ for $i=1,\dots,7$.
I decompose $\lambda D_F = \alpha \Gamma + C$, where $\Gamma \not\subset \Supp C$ and $\alpha \lem 1$ since bi-degree of $\lambda D_F$ equals $(2,2)$.
Thus
$$
7 < \sum_{i=1}^7 \mult_{P_i} \lambda D_F \lem 7 \alpha + C \cdot \Gamma = 7 \alpha + 6 - 6 \alpha = 6 + \alpha \lem 7,
$$
a contradiction.

Suppose there are curves $\Gamma_1,\Gamma_2$ as in the case \emph{(2)}.
Let $D$ be a general divisor in $\CM$ and denote $D_F = D\vert_F$.
Then $\mult_{P_i} \lambda D_F > 1$ for $i=1,\dots,8$.
I decompose $\lambda D_F = \alpha_1 \Gamma_1 + \alpha_2 \Gamma_2 + C$, where $\Gamma_1,\Gamma_2 \not\subset \Supp C$ and $\alpha_1 + \alpha_2 \lem 2$ since $\deg \lambda D_F = 4$.

 At most two points among $P_1,P_2,P_3,P_4$ coincide with points among $P_5,P_6,P_7,P_8$ since $\Gamma_1 \cdot \Gamma_2 = 2$.
Thus after renumbering points $P_i$ I may assume that
\begin{align*}
&\Gamma_1 \cap \Big( \{P_5,P_6,P_7,P_8\} \setminus \{P_1,P_2,P_3,P_4\} \Big)=\varnothing \quad \text{and}\\ 
&\Gamma_2 \cap \Big( \{P_1,P_2,P_3,P_4\} \setminus \{P_5,P_6,P_7,P_8\} \Big)=\varnothing.
\end{align*}
Thus
$$
8 < \sum_{i=1}^4 \mult_{P_i} \lambda D_F + \sum_{i=5}^8 \mult_{P_i} \lambda D_F \lem 4 \alpha_1 + C \cdot \Gamma_1 + 4 \alpha_2 + C \cdot \Gamma_2 = 4 + 2 \alpha_1 + 2 \alpha_2 \lem 8,
$$
a contradiction.
\end{proof}

\subsection{Supermaximal singularities}
To finish the proof of $\CA_5$-equivarian birational superrigidity I use the technique of supermaximal singularities. 
It has been introduced in \cite{Pukh123} for proving birational rigidity of del Pezzo fibrations of degrees $1$, $2$, and $3$.

First, I require a stronger version of Noether-Fano inequality.

\begin{Prop}[{\cite[Proposition~2.7]{Okada18}}] \label{supermaximal}
Let $\pi: X \to \MP^1$ be a del Pezzo fibration.
Suppose that we are given a non-square birational map $f\colon X \dasharrow Y$ to a Mori fiber space $\pi_Y \colon Y \to Z$ and let $\CM$ be a movable linear system associated to $f$.
Define numbers $\lambda \in \MQ_+$ and $l \in \MQ$ by the equivalence $\lambda \CM + K_X \sim l F$.
Suppose in addition that the pair $(X, \lambda \CM)$ is canonical at curves on $X$ and $l \gem 0$. 
Then there exist points $Q_1,\dots,Q_k$ of $X$ contained in distinct $\pi$-fibers and positive rational numbers $\gamma_1,\dots,\gamma_k$ with the following properties:
\begin{itemize}
\item
$(X,\lambda \CM - \sum \gamma_j F_j)$ is not canonical at $Q_1,\dots,Q_k$, where $F_i$ is the $\pi$-fiber containing $\Gamma_i$.
\item
$\sum_{j=1}^{k} \gamma_j > l$.
\end{itemize}
\end{Prop}

Let $D_1,D_2$ be general divisors in $\CM$ and denote their scheme theoretic intersection $Z = D_1 \cap D_2$.
I may decompose $Z$ into the horizontal and the vertical parts $Z = Z^h + Z^v$.
I may further decompose
$$
Z^v = \sum Z^v_i,\text{ where }\Supp Z_i^v \subset F_i.
$$

\begin{Lemma} \label{Deg-Bound}
Suppose $n \gem 2$, then
\begin{align*}
\lambda^2 \deg Z^h &= 8\\
\lambda^2 \deg Z^v &\lem 8l + 8
\end{align*}
\end{Lemma}
\begin{proof}
I compute
$$
\lambda^2 D_1 D_2 \equiv (-K_X + lF)^2 \equiv K_X^2 - 2 l F \cdot K_X.
$$
By Corollary \ref{Xn-Int}
$$
-2lF \cdot K_X \equiv 8l f
$$
and
$$
K_X^2 \equiv 8s + 24f - 8 n f
$$
These equivalences imply the statement of the lemma.
\end{proof}

\begin{Cor} \label{supermax-exists}
There is a fiber $F_j$ and a divisorial valuation $E_\infty$ of $\MC(X)$ such that $\deg \lambda^2 Z^v_j \lem 8 + 8 \gamma_i$, $a(E_\infty,X,\lambda\CM - \gamma_j F_j) < 0$, and the center $Q_j$ of $E_\infty$ is a point on $F_j$.
\end{Cor}

The valuation $E_\infty$ is called \emph{supermaximal singularity}.
By the previous section $Q_j \in \Delta$, where $\Delta$ is the $\CA_5$-invariant curve of degree $2$ on $F_j$.
From now on I denote $F_j$ by $F$, $\gamma_i$ by $\gamma$, $Q_j$ by $P$, and $Z^v_i$ by $Z_v$.

\subsection{Pukhlikov's inequality}

Consider the tower of blow ups realizing $E_\infty$
\begin{equation} \label{tower}
X_N \xrightarrow{\sigma_N}  \dots \xrightarrow{\sigma_2} X_1 \xrightarrow{\sigma_1} X_0 = X,
\end{equation}
that is $\sigma_i$ is the blow up of $X_{i-1}$ at the center $B_{i-1}$ of $E_\infty$ on $X_{i-1}$, $E_{i}$ is the exceptional divisor of $\sigma_i$, and $E_N = E_\infty$ as divisorial valuations of $\MC(X)$.

Let $A$ be an object on $X_i$, then I denote its proper transform on $X_j$, $j>i$ by $A^{(j)}$.
Denote
\begin{align*}
&K = \min \{ i \mid B_i\text{ is not a point} \},\\
&K^\prime = \min \{ i \mid B_i\not\in F^{(i)} \} \cup {K},\\
&K^{\prime\prime} = \min \{ i \mid B_i\not\in \Delta^{(i)} \},
\end{align*}
clearly $K \gem K^\prime \gem K^{\prime\prime}$.

There is an oriented graph associated to $(\ref{tower})$. 
It consists of $N$ vertices $v_i$, and there is an edge $v_i \to v_j$ if $i>j$ and $B_{i-1} \subset E_j^{(i-1)}$.
Denote by $p_i$ the number of paths from $v_N$ to $v_i$ and set $p_N = 1$.
Also set
\begin{align*}
\Sigma_0 = \sum_{i=1}^{K} p_i, \quad
\Sigma_0^\prime = \sum_{i=1}^{K^\prime} p_i, \quad
\Sigma_0^{\prime\prime} = \sum_{i=1}^{K^{\prime\prime}} p_i, \quad
\Sigma_1 = \sum_{i=K+1}^{N} p_i.
\end{align*}

Denote $\nu_i = \mult_{B_{i-1}} \lambda \CM^{(i-1)}$, then non-canonity of the pair $(X,\lambda \CM - \gamma F)$ is equivalent to
\begin{equation} \label{NFP}
\sum_{i=1}^N p_i \nu_i > 2 \Sigma_0 + \Sigma_1 + \gamma \Sigma_0^\prime.
\end{equation}

\begin{Th}[{\cite[Proposition~4.2]{Pukh123}~Pukhlikov's~inequality}]
$$
\lambda^2 \Big( \sum_{i=1}^K p_i \mult_{P_{i-1}} Z_h^{(i-1)} + \sum_{i=1}^{K^\prime} p_i \mult_{P_{i-1}} Z_v^{(i-1)} \Big) \gem 4 \Sigma_0 + 4 \gamma \Sigma_0^\prime + \frac{(\Sigma_1-\gamma\Sigma_0^\prime)^2}{\Sigma_0+\Sigma_1}.
$$
\end{Th}

I apply this inequality for each valuation in the orbit of $E$ to show a contradiction.


Denote the valuations in the $\CA_5$-orbit of $E_\infty$ by $E_{k,\infty}$, in particular $E_{1,\infty} = E_\infty = E_N$ as valuations, and let $\theta$ be the length of the orbit.
Recall that $a(E_{k,_\infty}, X,\lambda\CM - \gamma F)<0$ for $k=1,\dots,\theta$.
Thus there is a tower (\ref{tower}) for each valuation $k = 1,\dots,\theta$.
The graphs for the towers are identical for each $E_{k,N}$, thus the invariants $K$, $K^\prime$, $K^{\prime\prime}$, $N$, $p_i$, $\Sigma_0$, $\Sigma_0^\prime$, $\Sigma_0^{\prime\prime}$, and $\Sigma_1$ are the same for each $E_{k,N}$ as well.

From now on suppose that $\sigma_i$ is the blow up of $X_{i-1}$ at the centers of $E_{k,N}$.
Denote the center of $E_{k,N}$ on $X_i$ by $B_{k,i}$ and the exceptional divisor of $\sigma_i$ over $B_{k,i-1}$ by $E_{k,i}$.
Note that $\CM$ is $\CA_5$-invariant, therefore $\mult_{B_{k,i}} \CM^{(i)} = \mult_{B_{j,i}} \CM^{(i)} = \nu_{i+1}$ for any $1\lem j,k \lem \theta$, $0\lem i \lem N-1$.
Applying the Pukhlikov's inequality to $E_{k,N}$ and taking a sum I get the following inequality.

\begin{Cor}
\begin{equation} \label{G-Pukh}
\lambda^2 \sum_{k=1}^{\theta} \Big( \sum_{i=1}^{K} p_i \mult_{B_{k,i-1}} Z_h^{(i-1)} + \sum_{i=1}^{K^\prime} p_i \mult_{B_{k,i-1}} Z_v^{(i-1)} \Big)
\gem \theta \big( 4 \Sigma_0 + 4 \gamma \Sigma_0^\prime + \frac{(\Sigma_1-\gamma\Sigma_0^\prime)^2}{\Sigma_0+\Sigma_1} \big).
\end{equation}
\end{Cor}

To find a bound for the left-hand side, I decompose $\lambda^2 Z_v = \xi \Delta + C$, where $\Delta \not\in \Supp C$.

\begin{Lemma} \label{LHS-Lemma}
The following inequalities hold
\begin{enumerate}
\item
$$
\lambda^2 \sum_{k=1}^{\theta} \sum_{i=1}^{K} p_i \mult_{B_{k,i-1}} Z_h^{(i-1)} \lem 8 \Sigma_0,
$$
\item
$$
\sum_{k=1}^{\theta} \mult_{B_{k,0}} C \lem 8 + 8\gamma - 2 \xi,
$$
in particular, $\xi \lem 4 + 4\gamma$.
\end{enumerate}
\end{Lemma}
\begin{proof}
The inequality
$$
\lambda^2 \sum_{k=1}^\theta \mult_{B_{k,0}} Z_h \lem \lambda^2 Z_h \cdot F = 8.
$$
implies \emph{(1)}.

Assertion \emph{(2)} follows from
$$
\sum_{k=1}^{\theta} \mult_{B_{k,0}} C \lem C \cdot \Delta = \deg C \lem 8 + 8\gamma - 2 \xi.
$$
\end{proof}

\begin{Prop} \label{Pukh-bound}
The following inequality holds:
$$
8 \Sigma_0 > \theta \frac{(\Sigma_1-\gamma\Sigma_0^\prime)^2}{\Sigma_0+\Sigma_1}.
$$
\end{Prop}
\begin{proof}
Combining Lemma \ref{LHS-Lemma} with (\ref{G-Pukh}) I get
\begin{equation} \label{G-Pukh-s1}
8 \Sigma_0 + \theta \xi \Sigma_0^{\prime\prime} + (8 + 8\gamma - 2 \xi) \Sigma_0^\prime > \theta \Big(4\Sigma_0 + 4 \gamma \Sigma_0^\prime + \frac{(\Sigma_1-\gamma\Sigma_0^\prime)^2}{\Sigma_0+\Sigma_1}\Big).
\end{equation}
Observe that
$$
\theta \Sigma_0^{\prime\prime} > 2 \Sigma_0^\prime.
$$
Indeed, otherwise (\ref{G-Pukh-s1}) becomes 
$$
8 \Sigma_0 + (8 + 8 \gamma) \Sigma_0^\prime > 4 \theta \Sigma_0 + 4 \theta \gamma \Sigma_0^\prime \gem  48 \Sigma_0 + 48 \gamma \Sigma_0^\prime ,
$$
a contradiction.

I apply the bound $\xi \lem 4 + 4\gamma$ to (\ref{G-Pukh-s1}) to get the inequality
$$
8 \Sigma_0 + (4 + 4 \gamma) \theta \Sigma_0^{\prime\prime} > 4 \theta \Sigma_0 + 4 \theta \gamma \Sigma_0^\prime + \theta \frac{(\Sigma_1-\gamma\Sigma_0^\prime)^2}{\Sigma_0+\Sigma_1}.
$$
Since $\Sigma_0^\prime \gem \Sigma_0^{\prime\prime}$ I can further simplify
$$
8 \Sigma_0 + 4 \theta \Sigma_0^{\prime\prime} > 4 \theta \Sigma_0 + \theta \frac{(\Sigma_1-\gamma\Sigma_0^\prime)^2}{\Sigma_0+\Sigma_1},
$$
which implies the statement of the proposition.
\end{proof}

I use a different technique to get a bound contradicting this one.

\subsection{The technique of restricted multiplicities}
The technique I am using is inspired by \cite[Proof~of~Theorem~3.1]{Grin-V_1}.

Consider the linear system $\mathcal{S} \subset \abs{\gamma H+F}$ of divisors containing $\Delta$.
Let $S$ be general in $\mathcal{S}$ and let $D$ be general in $\CM$.
Denote $D_S = D\vert_S$ and $\mult_\Gamma D = \alpha$, then $\ord_\Gamma D_S = \alpha$.
Indeed, any $S \in \mathcal{S}$ has the equation $a(u,v) w + u p(x,y,z,w) = $ for some linear $p$ and $a$.
Thus $\ord_\Gamma S_1 \cap S_2 = 1$ for general $S_1,S_2 \in \mathcal{S}$.
It follows that $D_S = \alpha \Gamma + D_S^\prime$, where $\Gamma \not\subset \Supp D_S^\prime$.

\begin{Lemma} \label{blow-up-computations}
Let $B$ be a point or a smooth curve on a smooth threefold $X$.
Let $S$ be a surface on $X$, suppose $\WT B = B \cap S$ is a point and suppose $S$ is smooth at $\WT B$.
Let $\sigma \colon Y \to X$ be the blow up at $B$, let $E$ be its exceptional divisor, let $S_Y = \sigma^{-1} S$ and let $e = E \cap S_Y$.
Let $D$ be an effective divisor on $X$ and let $D_Y = \sigma^{-1} D$, then
$$
\mult_{\WT B} D\vert_S = \mult_{B} D + \ord_{e} D_Y\vert_{S_Y}.
$$
\end{Lemma}
\begin{proof}
Elementary calculations in local coordinates.
\end{proof}

Note that $\mult_{B_{k,i-1}} D_S^{(i-1)} = \mult_{B_{j,i-1}} D_S^{(i-1)}$ for general $D$ and $S$ and $1\lem j,k\lem \theta$.
Denote $\WT{\nu}_i = \lambda \mult_{B_{k,i-1}} D_S^{(i-1)}$, then one can bound $\nu_i$ using $\WT{\nu}_i$.

\begin{Lemma}
For any $i \lem K^{\prime\prime}$ we have
\begin{equation*} 
\nu_1 + \dots + \nu_{i} \lem \WT{\nu}_1 + \dots + \WT{\nu}_{i}.
\end{equation*}
\end{Lemma}
\begin{proof}
Let $e_{k,i} = S^{(i)} \cap E_{k,i}$ and denote $m_i = \ord_{e_{k,i}} \lambda D^{(i)}\vert_{S^{(i)}}$, then
$$
\lambda D^{(i)}\vert_{S^{(i)}} = \lambda D_S^{(i)} + m_1 e_{k,1}^{(i)} + \dots + m_{i-1} e_{k,i-1}^{(i)} + m_i e_{k,i}.
$$
Observe that $B_{k,i} \not \in E_{k,i-1}^{(i)}$ for $i < K^{\prime\prime}$ since $\Delta$ is smooth.
In particular $B_{k,i} \not\in e_{k,j}^{(i)}$ for $j < i$, thus
\begin{equation*}
\mult_{B_i} \lambda D^{(i)}\vert_{S^{(i)}} = \WT{\nu}_{i+1} + m_{i} \quad \text{for}~i\gem 1.
\end{equation*}
On the other hand by Lemma \ref{blow-up-computations}
$$
\mult_{B_i} \lambda D^{(i)}\vert_{S^{(i)}} = \nu_{i+1} + m_{i+1}.
$$
Thus I get the equalities
\begin{align*}
&\WT{\nu}_1 = \nu_1 + m_1\\
&\WT{\nu}_2 + m_1 = \nu_2 + m_2\\
&\quad\quad\quad \dots\dots\\
&\WT{\nu}_{K^{\prime\prime}} + m_{K^{\prime\prime} -1 } = \nu_{K^{\prime\prime}} + m_{K^{\prime\prime}},
\end{align*}
which together imply the statement of the lemma.
\end{proof}

\begin{Prop}
There is a bound
\begin{equation*}
\nu_1 + \dots + \nu_{K^{\prime\prime}} \lem K^{\prime\prime} + \frac{4}{\theta},
\end{equation*}
in particular
\begin{equation} \label{restricted-bound}
p_1 \nu_1 + \dots + p_N \nu_N \lem (\Sigma_0 + \Sigma_1) (1 + \frac{4}{\theta K^{\prime\prime}}),
\end{equation}
\end{Prop}
\begin{proof}
The curve $\Gamma$ is smooth at $P_i$, thus
\begin{align*}
p_j = p_1 \quad \text{and} \quad \WT{\nu}_{j} = \lambda \alpha + \lambda \mult_{\WT{B}_{i,j-1}} (D_S^\prime)^{(j-1)} \quad \text{for}~ j \lem K^{\prime\prime}
\end{align*}
On the other hand $S \cap F = \Gamma$, therefore
$$
\lambda \sum_{i=1}^{2\gamma} \sum_{j=1}^{L} \mult_{\WT{B}_{i,j}} (D_S^\prime)^{(j)} \lem \lambda D_S \cdot F = \lambda D \cdot F \cdot S = 4.
$$

Recall that $\CM$ is canonical at $\Gamma$ by Proposition \ref{Untwisting}, that is $\lambda \alpha \lem 1$.
Putting the bounds together I get the first inequality.
It follows that $\nu_{i} \lem 1 + \frac{4}{\theta K^{\prime\prime}}$ for $i > K^{\prime\prime}$ which implies the second inequality in the lemma.
\end{proof}

\begin{Corollary} \label{orbits-of-points}
Let $P$ be a point which is not $\CA_5$-fixed.
Then there is no supermaximal singularity at $P$.
\end{Corollary}
\begin{proof}
I combine (\ref{restricted-bound}) with (\ref{NFP}) to get a lower bound on $\Sigma_1$:
$$
\gamma \Sigma_0^\prime + 2 \Sigma_0 + \Sigma_1 < (\Sigma_0 + \Sigma_1) (1 + \frac{4}{\theta K^{\prime\prime}}),
$$
equivalently
$$
\theta K^{\prime\prime} \gamma \Sigma_0 + (\theta - 4) K^{\prime\prime}\Sigma_0 < 4\Sigma_1.
$$
In particular, since $K^{\prime\prime} \gem 1$ and $\theta \gem 12$ I get
$$
\Sigma_1 > 2 \Sigma_0 + 3 \gamma \Sigma_0^\prime.
$$

Applying this bound to Proposition \ref{Pukh-bound} I get
$$
8 \Sigma_0 - 4 \theta (\Sigma_0 - \Sigma^{\prime\prime}) > 12 \frac{4 \Sigma_0^2 + 8 \gamma \Sigma_0 \Sigma_0^\prime}{3 \Sigma_0+3 \gamma \Sigma_0^\prime} > 8 \Sigma_0,
$$
a contradiction.
\end{proof}

At last I am ready to prove Theorem \ref{Th-Conjugacy-classes}.

\begin{proof}[Proof of Theorem \ref{Th-Conjugacy-classes}]
Let $\chi \colon X_n \dasharrow Y$ be an $\CA_5$-equivariant birational map to a $\CA_5\MQ$-Mori fiber space $Y/Z$ and let $\CM$ be the mobile $\CA_5$-invariant linear system associated $\chi$.
Define the numbers $\lambda$ and $l$ by the equivalence $\lambda \CM + K_{X_n} \sim l F$.
By Lemma \ref{Untwisting} there is another birational model $\pi_b \colon X_b \to \MP^1$ and an $\CA_5$-equivariant square birational map $\varphi_b \colon X_n \dasharrow X_b$ such that the corresponding pair $(X_b,\lambda \CM_b)$ is canonical at curves and $\CA_5$-fixed points.
By Corollary \ref{orbits-of-points} the pair $(X_b,\lambda \CM_b)$ is also canonical at points not fixed by $\CA_5$-action.

By Proposition \ref{supermaximal} and Corollary \ref{supermax-exists} if the map $\chi \circ \varphi_b^{-1} \colon X_b \to Y$ is not $\CA_5$-equivariantly square, then there is a supermaximal singularity $E_\infty$.
But by Corollary \ref{orbits-of-points} there are no supermaximal singularities.
It follows that $\chi \circ \varphi_b^{-1}$ is $\CA_5$-equivariantly square and hence $\chi$ is $\CA_5$-equivariantly square.
\end{proof}

\section{The $\CA_5$-equivariant birational geometry of $X_1$}
In this section I describe some $\CA_5$-Mori fiber spaces $\CA_5$-equivariantly birational to the variety $X_1$.

Let $Q \subset \MP(W_3 \oplus I \oplus I)$ be a smooth $\CA_5$-invariant quadric.
Let $x,y,z$ be the coordinates on $W_3$ and $u,v$ be the coordinates on $I \oplus I$.
Let $\Gamma$ be the $\CA_5$-invariant point curve of degree $2$, then it has equations $(xz - y^2 = u = v = 0)$ after a change of coordinates on $W_3$.
Thus the blow up of $Q$ at $\Gamma$ is $\sigma_\Gamma \colon X_1 \to Q$.
On the other hand $Q$ is a quadric with an $\CA_5$-invariant point, hence it is $\CA_5$-equivariantly birational to $\MP(W_3 \oplus I)$.

Consider the blow up at the $\CA_5$-invariant point $\sigma_Y \colon Y_1 \to \MP(W_3 \oplus I)$, where
$$
Y_1 = \MP (\CO_{\MP(W_3)} \oplus \CO_{\MP(W_3)} (-1)).
$$
The $\MP^1$-bundle $\tau \colon Y_1 \to \MP(W_3)$ has many $\CA_5$-equivariantly square birational to it models.
For example, an $\CA_5$-invariant conic on the exceptional divisor of $\sigma_Y$ induces the $\CA_5$-equivariant elementary Sarkisov link $Y_1 \dasharrow Y_3$, where
$$
Y_3 = \MP (\CO_{\MP(W_3)} \oplus \CO_{\MP(W_3)} (-3)).
$$
Similarly, $Y_1$ is $\CA_5$-equivariantly birational to $Y_{2k+1}$ for any $k$.
Alternatively, we can take a fiber $f$ of $\tau$ and the blow up $\sigma \colon \WT Y \to Y_1$ at the orbit of $f$, it is easy to see that $\WT Y$ admits an $\CA_5$-equivariant $\MP^1$-bundle.

As we can see, $X_1$ has a rich $\CA_5$-equivariant birational geometry with the following $\CA_5$-Mori fiber spaces structures:
\begin{enumerate}
\item $\CA_5$-Fano variety $Q \in \MP(W_3 \oplus I \oplus I)$,
\item $\CA_5$-Fano variety $\MP (W_3 \oplus I)$,
\item $\CA_5$-del Pezzo fibration $\pi\colon X_1 \to \MP^1$,
\item $\CA_5$-conic bundle $\tau \colon Y_1 \to \MP(W_3)$.
\end{enumerate}

I believe that these are the only $\CA_5$-Mori fiber space structures.

Suppose $\chi\colon Q \to Y$ is a birational $\CA_5$-equivariant map to a $\CA_5\MQ$-Mori fiber space $\pi_Y \colon Y \to Z$.
Let $\CM_Q$ be the associated mobile linear system.
If $(Q,\lambda_Q \CM_Q)$ is not canonical at the $\CA_5$-invariant conic, then the map $\chi$ factors through $X_1$.
Let $\CM$ be the corresponding mobile linear system on $X_1$.
Elementary calculations show that $\lambda\CM \sim -K_{X_1} + l F$, where $l > 0$.
I believe that the results of Section \ref{rigidity-section} can be refined to show that $Y/Z$ is $\CA_5$-equivariantly square birational to $X_1/\MP^1$.

\begin{Quest}
Does there exist an $\CA_5\MQ$-Mori fiber space $\CA_5$-equivariantly birational to $Q$ which is not $\CA_5$-equivariantly square birational to $X_1/\MP^1$, $Y_1/\MP(W_3)$, $Q$, or $\MP(W_3 \oplus I)$?
\end{Quest}

\bibliographystyle{amsplain}
\bibliography{biblio}

\providecommand{\bysame}{\leavevmode\hbox to3em{\hrulefill}\thinspace}
\providecommand{\MR}{\relax\ifhmode\unskip\space\fi MR }
\providecommand{\MRhref}[2]{%
  \href{http://www.ams.org/mathscinet-getitem?mr=#1}{#2}
}
\providecommand{\href}[2]{#2}
\begin{thebibliography}{10}

\bibitem{ACPS19}
Hamid Ahmadinezhad, Ivan Cheltsov, Jihun Park, and Constantin Shramov,
  \emph{Double veronese cones with 28 nodes}, 2019.

\bibitem{Ch-Cubics}
I.\ Cheltsov, \emph{On singular cubic surfaces}, Asian J. Math. \textbf{13}
  (2009), no.~2, 191--214. \MR{2559108}

\bibitem{CS14}
Ivan Cheltsov and Constantin Shramov, \emph{Five embeddings of one simple
  group}, Trans. Amer. Math. Soc. \textbf{366} (2014), no.~3, 1289--1331.
  \MR{3145732}

\bibitem{A5-Book}
\bysame, \emph{Cremona groups and the icosahedron}, Monographs and Research
  Notes in Mathematics, CRC Press, Boca Raton, FL, 2016. \MR{3444095}

\bibitem{Corti00}
A.\ Corti, \emph{Singularities of linear systems and {$3$}-fold birational
  geometry}, Explicit birational geometry of 3-folds, London Math. Soc. Lecture
  Note Ser., vol. 281, Cambridge Univ. Press, Cambridge, 2000, pp.~259--312.

\bibitem{Corti95}
Alessio Corti, \emph{Factoring birational maps of threefolds after {S}arkisov},
  J. Algebraic Geom. \textbf{4} (1995), no.~2, 223--254. \MR{1311348}

\bibitem{DI09}
Igor~V. Dolgachev and Vasily~A. Iskovskikh, \emph{Finite subgroups of the plane
  {C}remona group}, Algebra, arithmetic, and geometry: in honor of {Y}u. {I}.
  {M}anin. {V}ol. {I}, Progr. Math., vol. 269, Birkh\"{a}user Boston, Boston,
  MA, 2009, pp.~443--548. \MR{2641179}

\bibitem{Grin-V_1}
M.~M. Grinenko, \emph{On the double cone over the {V}eronese surface}, Izv.
  Ross. Akad. Nauk Ser. Mat. \textbf{67} (2003), no.~3, 5--22. \MR{1992191}

\bibitem{KcA1}
Masayuki Kawakita, \emph{Divisorial contractions in dimension three which
  contract divisors to compound {$A_1$} points}, Compositio Math. \textbf{133}
  (2002), no.~1, 95--116. \MR{1918291}

\bibitem{SingPairs}
J\'{a}nos Koll\'{a}r, \emph{Singularities of pairs}, Algebraic
  geometry---{S}anta {C}ruz 1995, Proc. Sympos. Pure Math., vol.~62, Amer.
  Math. Soc., Providence, RI, 1997, pp.~221--287. \MR{1492525}

\bibitem{Kr18}
Igor Krylov, \emph{Birational geometry of del {P}ezzo fibrations with terminal
  quotient singularities}, J. Lond. Math. Soc. (2) \textbf{97} (2018), no.~2,
  222--246. \MR{3789845}

\bibitem{Okada18}
Takuzo Okada, \emph{On birational rigidity of singular del pezzo fibrations of
  degree 1},  (2018).

\bibitem{Pr12}
Yuri Prokhorov, \emph{Simple finite subgroups of the {C}remona group of rank
  3}, J. Algebraic Geom. \textbf{21} (2012), no.~3, 563--600. \MR{2914804}

\bibitem{GFanoI}
\bysame, \emph{{$G$}-{F}ano threefolds, {I}}, Adv. Geom. \textbf{13} (2013),
  no.~3, 389--418. \MR{3100917}

\bibitem{Pr15}
\bysame, \emph{On stable conjugacy of finite subgroups of the plane cremona
  group, ii}, The Michigan Mathematical Journal \textbf{64} (2015), no.~2,
  293–318.

\bibitem{Jordan-Cremona}
Yuri Prokhorov and Constantin Shramov, \emph{Jordan property for {C}remona
  groups}, Amer. J. Math. \textbf{138} (2016), no.~2, 403--418. \MR{3483470}

\bibitem{PrSh18}
\bysame, \emph{{$p$}-subgroups in the space {C}remona group}, Math. Nachr.
  \textbf{291} (2018), no.~8-9, 1374--1389. \MR{3817323}

\bibitem{Pukh123}
A.~Pukhlikov, \emph{Birational automorphisms of three-dimensional algebraic
  varieties with a pencil of del {P}ezzo surfaces}, Izv. Ross. Akad. Nauk Ser.
  Mat. \textbf{62} (1998), no.~1, 123--164.

\end{thebibliography}

\end{document}